\documentclass{amsproc}
\usepackage{amssymb,amsmath,amsthm,amscd}
\newtheorem{theorem}{Theorem}
\newtheorem{lemma}{Lemma}
\newtheorem{proposition}{Proposition}
\newtheorem{definition}{Definition}
\newtheorem{corollary}{Corollary}
\newtheorem{notation}{Notation}

\newtheorem*{prop}{Proposition}
\newtheorem*{thm}{Theorem}

\begin{document}

\title{The Complete Picard Vessiot Closure}
\author{Andy~R. Magid}
\address{Department of Mathematics\\
        University of Oklahoma\\
        Norman OK 73019}
\subjclass{12H05}
\date{September 3, 2007}
\maketitle

\begin{abstract}
Let $F$ be a differential field with field of constants $C$,
which we assume to be algebraically closed and of characteristic
$0$.  The complete Picard--Vessiot closure $F_\infty$ of $F$ is a
differential field extension of $F$ with the same constants $C$
as $F$, which has no Picard--Vessiot extensions, and is minimal
over $F$ with these properties. There is a correspondence between
subfields of $F_\infty$ and subgroups of its differential
automorphism group, which arises because $F_\infty$ comes from $F$
via repeated Picard--Vessiot extensions. This correspondence
persists for certain normal subfields of $F_\infty$, fields which
can be characterized independently of their embedding in
$F_\infty$.
\end{abstract}

\section{Introduction}

Let $F$ be a field of characteristic zero. Polynomials over $F$
have solutions in finite Galois extensions of $F$, and taking the
compositum of all these produces a minimal field extension
$\overline{F} \supseteq F$ in which all polynomials over $F$ have
a full set of solutions. Moreover $\overline{F}$ is a (generally
infinite) Galois extension of $F$, and there is a Galois
correspondence bijection between its lattice of subextensions and
the lattice of closed subgroups of the profinite group of
automorphisms of $\overline{F}$ over $F$. Of course $\overline{F}$
is the algebraic closure of $F$, and is itself algebraically
closed: every polynomial over $\overline{F}$ has a full set of
solutions in $\overline{F}$.

The situation with differential fields is analogous but more
complicated.

Let $F$ be a differential field of characteristic zero and with
algebraically closed field of constants $C$.

Linear differential equations over $F$ have solutions in
Picard--Vessiot extension fields of $F$, and taking a compositum
of all of these produces a minimal differential field extension
$F_1 \supseteq F$ in which every linear differential equation
over $F$ has a full set of solutions. It may happen that not
every linear differential equation over $F_1$ has a solution in
$F_1$, but by repeating the previous procedure we obtain an
extension field $F_2 \supseteq F_1$ in which the equations with
$F_1$ coefficients have solutions. Continuing in this fashion
produces a tower $F=F_0 \subseteq F_1 \dots F_i \subseteq F_{i+1}
\dots$ whose union $F_\infty$ has the property that every linear
differential equation over $F_\infty$ has a full set of solutions
in $F_\infty$. We will see that $F_\infty$ is minimal over $F$
with respect to this property. (We require that none of the
extensions of $F$ introduce a new constant.)

The field $F_\infty$ is a natural object of study when
considering differential equations starting from $F$ and their
solutions. Among the things we might be concerned with are the
elements of $F_\infty$, its automorphisms, and the lattice
structure of its differential subfields.

We will characterize the finitely generated differential
extensions of $F$ which embedd in $F_\infty$: these are the
differential fields which can be embedded in a finite tower of
Picard--Vessiot extensions over $F$ (a finite tower in which each
field is a Picard--Vessiot extension of its predecessor).

The group of differential automorphisms $G(F_\infty/F)$ of
$F_\infty$ over $F$ is the inverse limit of the groups
$G(F_i/F)$. Further, the restriction from $F_i$ to $F_{i-1}$
induces a surjection $G(F_i/F) \to G(F_{i-1}/F)$ whose kernel is
$G(F_i/F_{i-1})$. This latter is a (pro)affine (pro)algebraic
group, although neither $G(F_i/F)$ nor $G(F_\infty/F)$ need be.

If $H$ is any subgroup of $G(F_\infty/F)$ then the fixed field
$F_\infty^H$ is a differential subfield of $F_\infty$ which
contains $F$. We show that this observation has a full converse:
every differential subfield $K$ of $F_\infty$ containing $F$ is
the fixed field of some subgroup of $G(F_\infty/F)$, and hence of
$G(F_\infty/K)$. (This correspondence property -- all subfields
are fixed fields of subgroups -- is called ``semi--Galois" in the
literature.)

The field $F_\infty$ should be of special interest in the case
$F=C$, and we include an example in this case.

We fix throughout the notation of this introduction. If $E
\supset F$ is a differential field extension, we use $G(E/F)$ to
denote the group of differential automorphisms of $E$ over $F$. If
$S$ is a subset of $E$, we let $F\langle S \rangle$ denote the
smallest differential subfield of $E$ that contains both $F$ and
$S$. If $K$ is another differential subfield of $E$, then
$F\langle K \rangle$ is the field theoretic compositum $FK$. For
unreferenced definitions and notation regarding the
Picard--Vessiot theory, we refer to \cite{M} and \cite{PS}. For
general reference for Picard--Vessiot closure, we cite
\cite{jalg} and \cite{PS}.

\section{Universal Property} \label{S:universal}

It is important to note that, despite the notation, the symbol
$F_\infty$ includes choices: Picard--Vessiot closures of $F_i$
are not unique, although all are isomorphic \cite[Theorem,
p.6]{jalg}. Then, using the automorphism lifting property
\cite[Theorem, p. 7]{jalg}, it is clear that any two choices of
$F_\infty$ are also isomorphic. In this section, we will show
that $F_\infty$ has an intrinsic characterization.

In the iterative construction of $F_{i+1}$ as the Picard--Vessiot
closure of $F_i$ we have that the constants of $F_{i+1}$ are those
of $F_i$, that every linear homogeneous differential equation
over $F_i$ has a full set of solutions in $F_{i+1}$, and that
$F_{i+1}$ is generated as a differential field over $F_i$ by
solutions of such equations. We use these to characterize
$F_\infty$:

\begin{theorem} \label{T:charFinfty} The extension $F_\infty
\supseteq F$ satisfies
\begin{enumerate}
\item The constants of $F_\infty$ are those of $F$. \\
\item Every every linear homogeneous differential equation over
$F_\infty$ has a full set of solutions in $F_\infty$.\\
\item If $F_\infty \supseteq E \supseteq F$ is an intermediate differential subfield
such that every linear homogeneous differential equation over $E$
has a full set of solutions in $E$ then $E=F_\infty$.
\end{enumerate}
Moreover, any differential field $K \supseteq F$ with the above
properties is differentiably $F$ isomorphic to $F_\infty$.
\end{theorem}

\begin{proof} Since $F_\infty=\cup F_i$ and no new constants are
introduced in any $F_i$, the constants of $F_\infty$ are those of
$F$. If $Y^{(n)} + a_{n-1}Y^{(n-1)} + \dots + a_1Y^{(1)} +a_0Y=0$
is a differential equation over $F_\infty$ then for some $j$ we
have $a_i \in F_j$, $0 \leq i \leq n-1$. Then there is a full set
of solutions in $F_{j+1}$, and hence in $F_\infty$. Now suppose
$F_\infty \supseteq E \supseteq F$ is an intermediate field such
that every linear homogeneous differential equation over $E$ has
a full set of solutions in $E$. Then $F_0=F$ is contained in $E$,
and if we have $F_i$ contained in $E$ then there is a full set of
solutions in $E$ for every linear homogeneous differential
equation over $F_i$. Because $F_{i+1}$ is generated by such
solutions, it follows that $F_{i+1}$ is contained in $E$ as well.
Thus by induction $F_\infty$ is contained in $E$, so they
coincide.

Now suppose we have a differential field $K \supseteq F$ with the
three properties. We trivially have an $F$ differential embedding
$f_0:F_0 \to K$. Suppose we have an $F$ embedding $f_i:F_i \to
K$. Then we produce an $F$ embedding $f_{i+1}:F_{i+1} \to K$
extending $f_i$. We consider the set $\mathcal S$ of pairs $(f,E)$
where $F_{i+1} \supseteq E \supseteq F_i$ is an intermediate
differential field and $f$ extends $f_i$; $\mathcal S$ is ordered
in the obvious way and the union of any chain in $\mathcal S$ is
an upper bound for it in $\mathcal S$. Zorn's lemma then produces
a maximal pair $(f_\mu, E_\mu)$. If $E_\mu \neq F_{i+1}$, then
there is Picard--Vessiot extension of $F_i$ not contained in
$E_\mu$, say for the equation $Y^{(n)} + a_{n-1}Y^{(n-1)} + \dots
+ a_1Y^{(1)} +a_0Y=0$. Denote its left hand side by $L$. This
means that the Picard--Vessiot extension $E_L$ of $E_\mu$ for
$L=0$ contained in $F_{i+1}$ is proper. Let $V$ denote a full set
of solutions of $L=0$ in $E_L$ so that $E_L=E\langle V \rangle$
Let $f_\mu(L)$ denote $Y^{(n)} + f_\mu(a_{n-1})Y^{(n-1)} + \dots +
f_\mu(a_1)Y^{(1)} +f_\mu(a_0)Y$. By the properties of $K$,
$f_\mu(L)=0$ has a full set of solutions $W$. Then there is an
$F$ embedding $E_L=E_\mu\langle V \rangle \to f_\mu(E_\mu) \langle
W \rangle$ which extends $f_\mu$. This is contrary to the
maximality of $f_\mu$, and hence $E_\mu=F_{i+1}$ and we can take
$f_\mu$ for $f_{i+1}$. By induction, we obtain consistent $f_i$'s
for all $i$, and use them to define $f_\infty : F_\infty \to K$.
Since $f_\infty(F_\infty)$ has the property that every linear
homogeneous equation over it has a full set of solutions in it,
and $K$ is the only subfield of $K$ with this property, we have
that $f_\infty(F_\infty)=K$ and $f_\infty$ is the desired
isomorphism.
\end{proof}

We note that only role $F$ plays in the three conditions of
Theorem \ref{T:charFinfty} is in the third, where the minimality
condition is phrased over $F$. If $E$ is a differential subfield
of $F_\infty$ containing $F$, then $F_\infty$ also satisfies the
minimality condition over $E$. Consequently, $F_\infty$ is also
$E_\infty$, a fact we now record:

\begin{corollary} \label{C:EinftyIsFinfty} Let $E$ be a differential
subfield of $F_\infty$ with $F \subseteq E$. Then
$F_\infty=E_\infty$. In particular, all the fields $E_i$ can be
regarded as subfields of $F_\infty$.
\end{corollary}

The following lemma is a special case of Corollary
\ref{C:EinftyIsFinfty}  which we will need to use below.

\begin{lemma} \label{L:PVofK} Let $E$ be a differential
subfield of $F_\infty$ with $F \subseteq E$, and let $L $ be a
monic linear differential operator over $E$. Then there is a
differential subfield $E_L \subseteq F_\infty$ with $E \subseteq
E_L$ such that $E_L \supseteq E$ is a Picard--Vessiot extension
for $L$.
\end{lemma}

For the proof, we note that we have $E_1 \subseteq F_\infty$ by
the Corollary and $E_L \subseteq E_1$.

\section{Intermediate fields and the semi Galois correspondence} \label{S:inter}

\begin{definition} \label{D:intnormal} A differential subfield $E$ containing $F$ of $F_\infty$
will be called an \emph{intermediate field} (of $F_\infty$). An
intermediate field is called \emph{normal} (in $F_\infty$) if it
is preserved set--wise by every differential automorphism of
$F_\infty$ over $F$.
\end{definition}

Let $E$ be an intermediate field. By Corollary
\ref{C:EinftyIsFinfty} we have that $E_\infty=F_\infty$. It
follows that $G(F_\infty/E)=G(E_\infty/E)$ and since $E$ is the
fixed field of the latter, it is also the fixed field of the
former. Moreover, since differential automorphisms of $E$ extend
to $E_\infty$ \cite[Theorem, p. 7]{jalg}, differential
automorphisms of $E$ extend to $F_\infty$ (``automorphism lifting
property").

Now suppose $E \subset M \subset K$ are intermediate fields, with
$K$ normal. By the normality of $K$, restriction defines a group
homomorphism $G(F_\infty/E) \to G(K/E)$, which is surjective by
the automorphism lifting property. Since $E$ is the fixed field
of $G(F_\infty/E)$, it follows that $E$ is also the fixed field
of $G(K/E)$, and of course that $M$ is the fixed field of the
subgroup $G(K/M)$ of $G(K/E)$. In other words, the association $M
\mapsto G(K/M)$ is an injection from the set of differential
subfields of $K$ containing $E$ to the set of subgroups of
$G(K/E)$, with right inverse $H \mapsto K^H$.

An injection of the sort just observed is half of the Galois
correspondence, and has been termed a ``semi--Galois
correspondence" in the literature.

Once we have an internal description of when a differential field
extension $K \supseteq E$ can be embedded in $F_\infty$ so that
$E \supseteq F$ and $K$ is normal, we can state the semi--Galois
correspondence as a theorem for such extensions. This we will do
in Section \ref{S:intchar} below.

\section{Example} \label{S:example}

At this point we pause to consider an example. This example is
based on the one considered in \cite[pp 12-13]{jalg}. We let
$F=C(t)$ be the field of rational functions over $C$ with
derivation $t^\prime=1$. Inside $F_\infty$ we choose elements $y$
and $\{z_a \mid a \in C \}$ with $y^\prime = t^{-1}$ and
$z_a^\prime=((y+a)t)^{-1}$. (In fact $y \in F_1$ and $z_a \in
F_2$). Then $F\langle y \rangle=F(y)$, and $F\langle z_a \rangle
=F(y,z_a)$; we denote the latter by $E_a$. Note that $F \subset
F(y)$ and $F(y) \subset E_a$ are Picard--Vessiot extensions, and
$F \subset E_a$ is an iterated Picard--Vessiot extension: a
defining tower is $F \subset F(y) \subset E_a$. Let $\mathcal E
=\{E_a \mid a \in C \}$. (The compostium $E=F(y, \{z_a \mid a \in
C\})$ of $\mathcal E$ in $F_\infty$ is, by Lemma \ref{L:compIPV}
below, a LIPV extension of $F$.)

\medskip

The differential equations of which $y$ and $z_a$ are solutions
are, respectively, $tY^{\prime\prime}+Y^\prime=0$ and
$(y+a)tY^{\prime\prime}-(y+a+1)Y^\prime=0$. The coefficient $t$ of
the former is in the base field $F$ (and a solution of the
differential equation $Y^{\prime\prime}=0$), and the coefficient
$y+a+1$ of the latter is a solution of
$tY^{\prime\prime}+Y^\prime=0$, while the leading coefficient is a
solution of $tY^{\prime\prime\prime}+Y^{\prime\prime}=0$. Thus the
field $E$ is generated by solutions of differential equations
whose coefficients are solutions of differential equations.

\medskip

Let $\sigma \in G(F_\infty/F)$. For some notational convenience
below, we will write the action on the right. Since both $y$ and
$y^\sigma$ have derivative $t^{-1}$, we have
$y^\sigma=y+b(\sigma)$ for some $b(\sigma) \in C$. If also $\tau
\in G(F_\infty/F)$ then comparing $y^{\sigma\tau}$ and
$(y^\sigma)^\tau$ we see that $b(\sigma\tau)=b(\sigma)+b(\tau)$.
Also
$(z_a^\sigma)^\prime=(z_a^\prime)^\sigma=((y+a)t)^{-1})^\sigma=z_{a+b(\sigma)}^\prime$,
from which it follows that
$z_a^\sigma=z_{a+b(\sigma)}+c(\sigma,a)$ for some $c(\sigma,a)\in
C$. Comparing $z_a^{\sigma\tau}$ and $(z_a^\sigma)^\tau$ we find
that $c(\sigma\tau,a)=c(\sigma,a+b(\tau))+c(\tau,a)$.

These calculations show that every $\sigma$ in $G(F_\infty/F)$
carries $E$ to itself and hence that $E$ is a normal LIPV
extension of $F$ (see defintion \ref{D:LIPV} below). (In fact,
they can be used to show that $E$ is the smallest normal
extension of $F$ which contains $E_0$.)

The extensions $F(y) \subset E_a$ are Picard--Vessiot as
previously noted; this makes the compositum extension $F(y)
\subset E$ a locally Picard--Vessiot extension \cite[p. 153]{P}
(or infinite Picard--Vessiot extension \cite[p. 2]{jalg}). The
differential Galois group $G(E_a/F(y))$ is the additive algebraic
group $\mathbb G_a$ (more precisely, its $C$ points $\mathbb
G_a(C)$ are), and the differential Galois group $G(E/F(y))$ is the
proalagebraic group $\prod_{c \in C}\mathbb G_a$. (One can see the
latter by regarding the product as the set $\text{Map}(C,\mathbb
G_a)$ of (all) functions from $C$ to $C$ and map $G(E/F(y)$ to it
by $\sigma \mapsto (a \mapsto c(\sigma,a))$. If $\sigma \in
G(E/F(y))$ then $\sigma(y)=y$ so $b(\sigma)=0$ and thus $G(E/F(y))
\to \text{Map}(C,\mathbb G_a)$ is a homomorphism.)

It follows that the extension $F \subset E$ goes in steps $F
\subset F(y)$ and $F(y) \subset E$ with each step (locally)
Picard--Vessiot, but of course the extension $F \subset E$ is not.

The group $G(E/F)$ is the semidirect product $\text{Map}(C,\mathbb
G_a) \rtimes \mathbb G_a$ where $G_a=C$ acts on
$\text{Map}(C,\mathbb G_a)$ by $a\cdot f (x)=f(x+a)$: we have
$G(E/F) \to \text{Map}(C,\mathbb G_a)\rtimes \mathbb G_a$ by
$\sigma \mapsto (c(\sigma, \cdot), b(\sigma))$.

As abstract groups the semidirect product $\text{Map}(C,\mathbb
G_a) \rtimes \mathbb G_a$ is the regular wreath product $\mathbb
G_a \wr_r \mathbb G_a$.

The semidirect product also has an algebraic (pro)variety
structure: it is the product of the proaffine variety $\prod_{c
\in C}\mathbb G_a$ and the variety $\mathbb G_a$. This has as
coordinate ring the polynomial ring $C[\{X_c \vert c \in C \}, Y]$
where (reverting to semidirect product and the functional notation
for the first factor) $X_c((f,a))=f(c)$ and $Y((f,a))=a$. We use
the ``dot" notation for the action on functions on a group coming
from translation actions on the group: $\phi\cdot a (b)=\phi(ab)$
and $b\cdot \phi(a)=\phi(ab)$ \cite{HM}. Then $X_c\cdot
(f,a)((g,b))=X_c((f,a)(g,b))=f(c)+g(a+c)$ so that
\[
X_c \cdot (f,a)=f(c)+X_{c+a}.
\]

Also, $Y\cdot (f,a)((g,b))=Y((f,a)(g,b))=a+b$, so that
\[
Y\cdot (f,a)=a+Y.
\]

The displayed formulae show that left translation on
$\text{Map}(C,\mathbb G_a) \rtimes \mathbb G_a$ by the fixed
element $(f,a)$ is a morphism of proaffine varieties.

Since $(g,b) \cdot X_c((f,a))=X_c((f,a)(g,b))=f(c)+g(a+c)$,
however, we can see that right translation is not necessarily a
morphism. Consider the case where $C=\mathbb C$ is the complex
numbers and $g=\exp$ the exponential function. Then
\[
((\exp,0)\cdot X_0-x_0)((f,a))=\exp(a).
\]
It is clear that $(f,a) \mapsto \exp(a)$ is not in the polynomial
ring $C[\{X_c \vert c \in C \}, Y]$.

\section{Characterization of intermediate fields} \label{S:intchar}

We we now return to an internal characterization of intermediate
fields, and normal intermediate fields.

We begin with finitely generated extensions:

\begin{definition} \label{D:IPV}A differential extension $E$ of $F$ is an
\emph{iterated Picard--Vessiot} (IPV) extension if there is a
chain of  differential subfields $F=E_0 \subseteq E_1 \dots
\subseteq E_n=E$ such that  for each $i$ $E_{i+1}$ is a
Picard--Vessiot extension of $E_i$. We call the fields $E_0, E_1,
\dots, E_n$ a \emph{defining tower} for $E$.
\end{definition}

The notation for defining towers of IPV extensions overlaps that
for the tower of Picard--Vessiot closures. It will be clear from
context, or noted, which one is intended.

Note that a Picard--Vessiot extension of an IPV extension of $F$
is also an IPV extension of $F$. Since Picard--Vessiot extensions
of $F$ are finitely generated as fields, iterated Picard--Vessiot
extensions of $F$ are finitely generated over $F$.

\begin{theorem} \label{T:FGeqIPV} Let $E$ be a differential
subfield of $F_\infty$ finitely generated over $F$. Then $E$ is
contained in an iterated Picard--Vessiot extension of $F$.
Conversely, if $E \supseteq F$ is a subfield of an iterated
Picard--Vessiot  extension then there is a differential embedding
of $E$ over $F$ into $F_\infty$.
\end{theorem}

\begin{proof} Suppose $X$ is a finite subset of $F_\infty$.
Let $k$ be the least $i$ such that $X \subset F_{i}$. Let $E$ be
the differential field generated over $F$ by $X$ (note that $E
\subseteq F_i$). We prove that $E$ is contained in an IPV by
induction on $k$. If $k=0$ then $X \subset F$ and the assertion is
trivial. So suppose $k>0$ and that the assertion holds by
induction for values less than $k$.  Assume $X \subset F_k$. Since
$X$ is finite, there is a Picard--Vessiot extension $M \supseteq
F_{k-1}$ such that $X \subset M$. Each element of $M$ is a ratio
of elements which satisfy differential equations over $F_{i-1}$.
Let $X_0$ be a finite set of elements of $F_i$ whose ratios
contain $X$ and such that each element $x \in X_0$ satisfies a
monic linear homogeneous differential operator $L_x$ with
coefficients in $F_{i-1}$. Let $Y$ be the union of the sets of
coefficients of the $L_x$. Note that $Y$ is finite. Since $Y
\subset F_{i-1}$, by induction the differential field $N$
generated over $F$ by $Y$ is IPV. By Lemma \ref{L:PVofK}, for
each $x \in X_0$ there is a Picard--Vessiot extension $E_x$ of $N$
for $L_x$ contained in $F_\infty$. Note that $x \in E_x$. Let
$E_{X_0}$ be the compositum of the $E_x$'s. By \cite[Cor. 10 p.
156]{P}, $E_{X_0} \supseteq N$ is Picard--Vessiot (so $E_{X_0}
\supseteq F$ is IPV), and $X_0 \subset E_{X_0}$ so also $X \subset
E_{X_0}$ and hence $E \subseteq E_{X_0}$. This proves the
inductive step and shows that finitely generated extensions of $F$
in $F_\infty$ are contained in IPV subextensions.

Conversely, suppose that $E \supseteq F$ is an IPV extension with
defining tower $F=E_0 \supseteq E_1 \supseteq \dots \supseteq
E_n$. We prove by induction on $n$ that $E$ embeds over $F$ into
$F_\infty$. For $n=0$, this is trivial. So suppose $n>0$ and that
the result holds for all smaller values. Then $E_{n-1}$ embeds in
$F_\infty$ over $F$, and we may assume that in fact this embedding
is inclusion. Suppose $E_n \supseteq E_{n-1}$ is Picard--Vessiot
for an operator $L$ over $E_{n-1}$. By Lemma \ref{L:PVofK} there
is a Picard--Vessiot extension $M$ of $E_{n-1}$ for $L$ contained
in $F_\infty$, and by uniqueness of Picard--Vessiot extensions
\cite[Theorem 3.5 p. 25]{M} $E_n$ and $M$ are isomorphic over
$E_{n-1}$; this isomorphism is the desired embedding, and the
result follows by induction. And if IPV extensions of $F$ embed in
$F_\infty$, so do their subfields.
\end{proof}

Now we expand the definition to locally iterated Picard--Vessiot
extensions:

\begin{definition} \label{D:LIPV} A differential extension $E$ of $F$ is a
\emph{locally iterated Picard--Vessiot} (LIPV) extension if every
finite subset of $E$ belongs to an iterated Picard--Vessiot
subextension of $F$ contained in $E$.
\end{definition}

The analog of Theorem \ref{T:FGeqIPV} applies for LIPV
extensions. We will need to use the following lemma about
composita of IPV extensions.

\begin{lemma} \label{L:compIPV} Let $K \supseteq F$ be a
differential extension with the same contants as $F$, and let
$\mathcal I$ be a set of locally iterated Picard--Vessiot
subextensions of $F$ in $K$. Then the compositum $E$ of the
fields in $\mathcal I$ is locally iterated Picard--Vessiot.
\end{lemma}

\begin{proof} Since by definition a LIPV extension of $F$ is the
compositum of its IPV subextensions, it suffices to prove the
lemma in the case where the elements of $\mathcal I$ are IPV
extensions. Moreover, if the lemma is true in the case that
$\mathcal I$ is finite, then $E$ will be the union of its IPV
subextensions, and hence LIPV. Hence it suffices to prove the
lemma when $\mathcal I$ consists of two subfields $M$ and $N$ with
defining towers $F=M_0 \subseteq \dots \subseteq M_m=M$ and $F=N_0
\subseteq \dots \subseteq N_n=N$. We prove that the compositum
$M_mN_n$ is IPV by induction on $n$. In the case $n=0$, $M_mF=M_m$
is IPV over $F$ by assumption. So suppose it is true for $n-1$, so
that $M_mN_{n-1}$ is IPV over $F$. Then since $N_{n-1} \subseteq
N_n$ is Picard--Vessiot, so is $M_mN_{n-1} \subseteq M_mN_n$ (here
is where we use that $K$ has no new contants). Thus $F \subseteq
M_mN_n=MN$ is IPV as desired.
\end{proof}

\begin{theorem} \label{T:eqLIPV} Let $E$ be a differential
subfield of $F_\infty$. Then $E$ is contained in a locally
iterated Picard--Vessiot extension of $F$. Conversely, if $E
\supseteq F$ is a subfield of an locally iterated Picard--Vessiot
extension then there is a differential embedding of $E$ over $F$
into $F_\infty$.
\end{theorem}

\begin{proof} Let $E$ be a differential subfield of $F_\infty$.
Every finite subset of $E$ belongs to an IPV extension of $F$ by
Theorem \ref{T:FGeqIPV}. Let $\mathcal I$ be the set of these
subfields; by Lemma \ref{L:compIPV} their compositum is LIPV.

Now suppose that $E$ is a LIPV extension of $F$. We construct an
embedding of $E$ into $F_\infty$ over $F$. Considering pairs of
subfields of $E$ and embeddings we can use Zorn's lemma to find a
differential subfield $E^0$ of $E$ containing $F$ and an embedding
$f^0:E^0 \to F_\infty$ over $F$ which can't be further extended.
Suppose $E^0 \neq E$, and $a \in E$, $a \notin E^0$. There is an
IPV extension $K$ of $F$ in $E$ containing $a$, and then there is
a set in the defining tower of $K$ with $K_{i-1} \subseteq E^0$
but $K_i \nsubseteq E^0$. Consider the field $E^1=E^0K_i$: this is
a Picard--Vessiot extension of $E^0$ with respect to some operator
$L$ over $E^0$. Let $f^0(L)$ denote the corresponding operator
over $f^0(E^0)$ obtained by applying $f^0$ to the coefficients of
$L$. By Lemma \ref{L:PVofK} we know that there is a
Picard--Vessiot extension $M$ of $f^0(E^0)$ for $f^0(L)$ contained
in $F_\infty$, and then by \cite[Proposition 18 p.161]{P} that
$f^0$ extends to an isomorphism $f^1$ from $E^1$ to $M$. The field
$E^1$ and the map $f^1:E^1 \to F_\infty$ contradict the maximality
of $E^0,f^0$, and so we conclude that $E^0=E$. Since LIPV
extensions embedd over $F$ into $F_\infty$, so do their subfields.
\end{proof}

The extension argument used in the proof of Theorem
\ref{T:eqLIPV} actually proves a bit more. Rather than simply
show the existence of an embedding into $F_\infty$ over $F$ of the
LIPV extension $E$, it shows that any embedding $f^o:E^o \to
F_\infty$ of a subextension $E^o$ of $E$ can be extended to all of
$E$ (we use a Zorn's Lemma argument to make $E^0$ a maximal
subfield containing $E^o$ on which the original embedding $f^o$
extends to an embedding $f^0$ and then argue as in the theorem
that $E^0=E$).

Now we turn to the definition of normality:

\begin{definition} \label{D:NormalLIPV} A locally iterated
Picard--Vessiot extension $E$ of $F$ is \emph{normal} if whenever
$K \supseteq F$ is an extension with no new constants then all
embeddings of $E$ in $K$ over $F$ have the same image.
\end{definition}

Note that Picard--Vessiot extensions always have the normality
property \cite[Proposition 3.3 p. 24]{M}. An interated
Picard--Vessiot extension which is not Picard--Vessiot will not be
normal.

We will need the following result about extensions of
isomorphisms, which is also of independent interest:

\begin{proposition} \label{P:isoTOauto} Let $K^1$ and $K^2$ be
locally iterated Picard--Vessiot extensions of $F$ inside
$F_\infty$, and suppose $\tau:K^1 \to K^2$ is an $F$ differential
isomorphism. Then there is an $F$ differential automorphism
$\sigma$ of $F_\infty$ which restricts to $\tau$ on $K^1$.
\end{proposition}

\begin{proof} Suppose $\sigma:F_\infty \to F_\infty$ is an $F$
differential embedding. Assume $\sigma$ is not onto. Let
$E=F_\infty$, let $E^0=\sigma(F_\infty)$, and let
$f^0=\sigma^{-1}:E^0 \to F_\infty$. The extension argument in
Theorem \ref{T:eqLIPV} (in the notation of that lemma) gives a
field $E^1$, $F_\infty \supseteq E^1 \supset E^0$ and an embedding
$f^1:E^1 \to F_\infty$. If $y \in E^1$, $y \notin E^0$, and
$x=\sigma(f^1(y))$, then $y-x \neq 0$ but $f^1(y-x)=0$. This
contradiction shows that any such $\sigma$ is onto. Thus to prove
the Proposition, it suffices to show that the embedding $\tau:
K^1 \to F_\infty$ extends to an embedding $\sigma:F_\infty \to
F_\infty$. Since $F_\infty$ is LIPV over $F$, this again follows
by the extension argument of Theorem \ref{T:eqLIPV}.
\end{proof}

\begin{theorem} \label{T:normal=>normal} Let $E$ be a locally iterated
Picard--Vessiot extension of $F$ contained in $F_\infty$. Then
the following conditions are equivalent:
\begin{enumerate}
\item Every differential automorphism of $F_\infty$ over $F$
carries $E$ to itself.  \label{nn1}\\
\item For any no new constants extension $K$ of $F$, all
differential embeddings $E \to K$ over $F$ have the same image.
\label{nn2}
\end{enumerate}
\end{theorem}

\begin{proof} Suppose that $\sigma(E)=E$ for all $\sigma \in
G(F_\infty/F)$, and let $\tau_i:E \to K$, $i=1,2$, be embeddings
over $F$ into a no new constants extension. By Lemma
\ref{L:compIPV} the compositum $\tau_1(E)\tau_2(E)$ is a LIPV
extension of $F$, and thus by Theorem \ref{T:eqLIPV} embedds over
$F$ into $F_\infty$. We compose the $\tau_i$ with this embedding,
so we may assume that $K=F_\infty$. Now consider $\tau_1(E)
\subseteq F_\infty$ and the embedding
$\tau=\tau_2\tau_1^{-1}:\tau_1(E) \to F_\infty$. By Proposition
\eqref{P:isoTOauto}, there is $\sigma \in G(F_\infty/F)$ which
extends $\tau$. Since $\sigma(\tau_1(E))=\tau_1(E)$ by
assumption, it follows that $\tau_1(E)=\tau_2(E)$, and so
\eqref{nn1} implies \eqref{nn2}.

Now suppose that for any no new constants extension $K$ of $F$,
all differential embeddings $E \to K$ over $F$ have the same
image, and let $\sigma \in G(F_\infty/F)$. Let $\tau_i:E \to
F_\infty$ be defined as follows: let $\tau_1$ be the restriction
of $\sigma$ to $E$ and let $\tau_2$ be the inclusion. Since by
assumption $\tau_1(E)=\tau_2(E)$, we have $\sigma(E)=E$ and so
\eqref{nn1} implies \eqref{nn2}.
\end{proof}

Theorem \ref{T:normal=>normal} shows that the above definition of
normal, Definition \ref{D:NormalLIPV}, coincides with that for
intermediate fields, Definition \ref{D:intnormal}. We have
already observed the semi--Galois correspondence for the latter,
which means we now have proven, and state as a theorem, the
semi--Galois correspondence for the former:

\begin{theorem}[Fundamental Theorem for Normal LIPV Extensions]
\label{T:Fundamental} Let $K$ be a normal locally iterated
Picard--Vessiot extension of $F$. Let $\mathcal E$ denote the set
of differential subfields of $K$ containing $F$ (``set of
intermediate fields") and let $\mathcal H$ denote the set of
closed subgroups of $G(K/F)$. Then correspondences
\begin{align}
\mathcal E \to \mathcal H \qquad &\text{ and } \qquad \mathcal H
\to \mathcal E \notag \\
&\text{ by } \notag \\
E \mapsto G(K/E) \qquad &\text{ and } \qquad H \mapsto K^H \notag
\end{align}
are bijections.
\end{theorem}

\end{document}